\documentclass[a4paper ,english]{article}
\usepackage{a4wide}
\usepackage[T1]{fontenc}
\usepackage{amssymb}
\usepackage{amsmath}
	\allowdisplaybreaks
\usepackage{amsfonts}
\usepackage{amsthm}
\usepackage{enumerate}
\usepackage{stmaryrd}
\usepackage{subcaption}
\usepackage{graphicx}\usepackage{xcolor}
\newcommand{\defgl}{\mathrel{\mathop:}=}
\newcommand\eqclass[1]{\ensuremath{\llbracket #1\rrbracket}}
\newcommand{\N}{{\ensuremath{\mathbb{N}}}}
\DeclareMathOperator{\Anc}{Anc}

\newcounter{mycounter}[section]%

\newtheorem{theorem}[mycounter]{Theorem}
\newtheorem{lem}[mycounter]{Lemma}
\newtheorem{prop}[mycounter]{Proposition}

\newtheorem{defi}[mycounter]{Definition}
\newtheorem{kor}[mycounter]{Corollary}
\newtheorem{bsp}[mycounter]{Example}
\newtheorem{rem}[mycounter]{Remark}
\newtheorem*{assumpt}{Assumption}
	
\usepackage{hyperref} 
	\hypersetup{pdftitle={Martin boundary theory on inhomogenous fractals}, pdfauthor={Uta Freiberg, Stefan Kohl}}

\title{Martin boundary theory on inhomogenous fractals}
\author{%
Uta Freiberg\footnote{Institute for Stochastics and Applications, University of Stuttgart, Pfaffenwaldring 57, 70569 Stuttgart, Germany, E-mail: \href{mailto:Freiberg@mathematik.uni-stuttgart.de}{\texttt{Freiberg@mathematik.uni-stuttgart.de}}}, 
Stefan Kohl\footnote{Institute for Stochastics and Applications, University of Stuttgart, Pfaffenwaldring 57, 70569 Stuttgart, Germany, E-mail: \href{mailto:Stefan.Kohl@mathematik.uni-stuttgart.de}{\texttt{Stefan.Kohl@mathematik.uni-stuttgart.de}}}%
}
\begin{document}
\maketitle
\begin{abstract}
 We want to consider fractals generated by a probabilistic iterated function scheme with open set condition and we want to interpret the probabilities as weights for every part of the fractal. In the homogenous case, where the weights are not taken into account, Denker and Sato introduced in 2001 a Markov chain on the word space and proved, that the Martin boundary is homeomorphic to the fractal set. Our aim is to redefine the transition probability with respect to the weights and to calculate the Martin boundary. As we will see, the inhomogenous Martin boundary coincides with the homogenous case.
\end{abstract}
\bigskip

\bigskip
%
%
%
%
\section{Introduction}
The general idea of Martin boundary may be first introduced by Martin in \cite{Martin1941} and was then extended by Doob \cite{Doob1959} and Hunt \cite{Hunt1960} by investigating the behavior of Markov chains in limit. Further articles and books \cite{Doob1984, Dynkin1969, KSK1976, Sawyer1997, Woess2000, Woess2009} followed this idea and tied the connection to harmonic analysis. This comes from the fact, that a harmonic function $h$ on $\mathcal W$ has an integral representation by
\begin{equation}
 h(w)=\int_{\mathcal M} k(w,\xi)\mathrm d\mu(\xi), \qquad w\in\mathcal W\nonumber
\end{equation}
where $\mathcal M$ is the Martin boundary, $k$ the Martin kernel and $\mu$ a Borel measure.

Denker and Sato came up with the idea, to describe fractals through Martin boundary theory. They studied in several papers \cite{DS1999, DS2001, DS2002} the description of the Sierpi\'nski gasket and proved that the Sierpi\'nski gasket is homeomorphic to the Martin boundary of a Markov chain on the word space. They defined so called strongly harmonic functions and an analogous of the Laplacian on the Martin boundary. They compared their results with the general approach of Kigami \cite{Kigami1993, Kigami2001} and showed that both definitions of harmonic functions coincide. This idea was later picked up by Lau and coauthors \cite{JuLauWang2012, LauNgai2012} and they proved that the results hold for all fractals satisfying the (OSC) and some assumptions on the Markov chain.\cite{LauWang2015}.

It is a natural question, if one can extend this idea to fractals, which are maybe not so regular. For example, if one modifies the Markov chain to be non-isotropic. This was done by Kesseböhmer, Samuel and Sender in \cite{KSS2017} for the Sierpi\'nski gasket and they showed that the Martin boundary is still homeomorphic to the fractal.

Our aim is to examine the case, where we extend the IFS of the fractal by weights. This leads to a probabilistic iterated function scheme and simultaneously to a self-similar measure. In order to connect the mass distribution with the Markov chain, we have to adapt the transition probability such that it fits to the weights. As a consequence, the Green function and the Martin kernel change. We can show, that this has no influence on the Martin boundary in the inhomogenous case and the Martin boundary is still homeomorphic to the fractal.

This article is structured as follows. In Section \ref{sec:preliminaries}  we introduce the notation, some general facts about fractals and the mass distribution. Further we link up iterated function schemes with the word space.\\
In Section \ref{sec:p} we want to define a new type of transition probability on the word space, which takes the mass distribution into account. We are then able to define the Markov chain, the Green function and a new function $q$, which helps us to understand much better the behavior of the Green function. In Section \ref{sec:k} we define the Martin kernel $k$ and observe some useful properties  of $k$. A essential part in this section (and for the whole paper) is Theorem \ref{theorem:156}, which gives us the opportunity to calculate the Martin kernel in the inhomogenous case through the Martin kernel of the homogenous case. Based on $k$ we are then able to define the Martin metric $\rho$ on the word space, which enables us to define in Section \ref{sec:MB} the Martin boundary. In Theorem \ref{theo:MBgleich} we show, that the Martin boundary in the inhomogenous and the homogenous case are equal.
%
%
%
%
\section{Preliminaries}
\label{sec:preliminaries}
We want to follow mainly the notation of Denker and Sato, but in a more general setting and suppose, that the idea of fractals as Martin boundary is roughly known. For this, we refer the reader to \cite{DS2001, DS2002, LauWang2015}.
\medskip

Consider an iterated function scheme (shortened IFS) $\{S_1,\dots, S_N\} : D\subseteq \mathbb R^d\to D$ with $N$ finite and where $S_i$ are similitudes, i.e. $|S_i(x)-S_i(y)|=c_i|x-y|$ with $0<c_i<1$. Due Hutchinsons theorem \cite{Hutchinson1981} it holds, that there is a unique, non-empty compact invariant subset $K\subset \mathbb R^d$ fulfilling 
\begin{equation*}
	K=S(K)\defgl\bigcup_{i=1}^NS_i(K).
\end{equation*}

The set $K$ is called attractor of the IFS $\{S_i\}_{i=1}^N$ and we want to assume for the whole paper, that $K$ is connected. If $K$ would not be connected, we can still do the whole calculus, but it would be quiet uninteresting, since we would be later unable to define harmonic functions. %
Additionally we want to assume, that the IFS satisfies the open set condition, appreviated by (OSC). This means there exists a non-empty bounded open set $\mathcal O\subset\mathbb R^d$ such that $\bigcup_{i=1}^NS_i(\mathcal O)\subset \mathcal O$ with the union disjoint.
\medskip

The IFS respectively the pre-fractal can be described by the word space. For this, consider the alphabet $\mathcal A=\{1,\dots,N\}$ of $N$ letters and the word space 
\begin{equation*}
	\mathcal W \defgl \bigcup_{n\geq 1} \mathcal A^n\cup\{\emptyset\}
\end{equation*}
where $\emptyset$ is the empty word. Denote by $\mathcal W^\star$ the set of all infinite $\mathcal A$-valued sequences $w=w_1w_2\cdots$ and by $w\big|_n = w_1\cdots w_n$ the restriction to the first $n$ letters of $w\in\mathcal W^\star$.\\
\begin{figure}[ht]%
	\centering%
	\includegraphics[page=1,width=7.5cm]{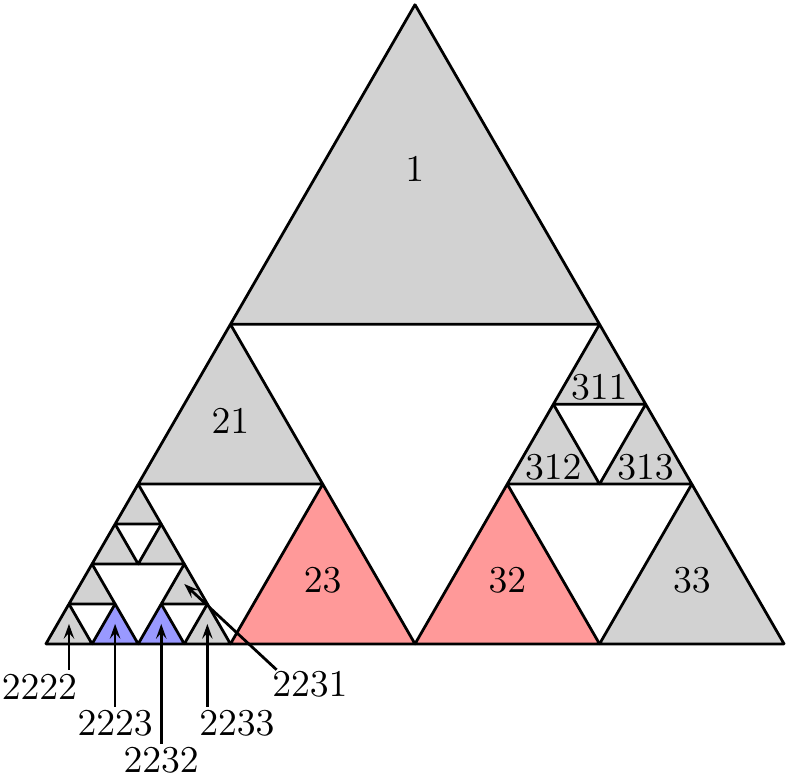}%
	\caption{Different cells of the Sierpi\'nski gasket and some equivalent words highlighted by the same color.}%
	\label{fig:CellsSG}%
\end{figure}%
For $w=w_1\cdots w_n\in\mathcal W$ with $w_i\in\mathcal A$ we want to define $\tau(w)\defgl w_n$ which is the last letter of the word $w$, the parent $w^-$ of $w$ by $w^- \defgl w_1\dots w_{n-1}$ and the length of $w$ through $|w|=n$. For two words $v, w\in\mathcal W$ we define $d(v,w)= |w|-|v|$.\\
The product $vw$ of two words $v=v_1 v_2\cdots v_m$ and $w=w_1 w_2\cdots w_n$ is defined by 
\begin{equation*}
	vw \defgl v_1 v_2 \cdots v_m w_1 w_2 \cdots w_n.
\end{equation*}
For the empty word $\emptyset$ it holds, that $|\emptyset| = 0$ and $w\emptyset = \emptyset w = w$ with $w\in\mathcal W$.
\medskip

To establish a connection between the IFS and the word space, we define $S_w(E)\defgl S_{w_1}\circ\dots\circ S_{w_n}(E)= S_{w_1}(\cdots(S_{w_n}(E)))$ for $E\subset \mathbb R^d$ and we can consider words as cells of a fractal and vice versa. %

In Figure \ref{fig:CellsSG} this is shown on the Sierpi\'nski gasket, where the upper triangle is coded by $1$ and respectively generated by $S_1$, the bottom left is decoded by $2$ and the bottom right triangle is noted by $3$.\\
An essential part is to define, when two words are equivalent. The idea behind this is to identify two infinite sequences, which decode the same point in the fractal. This should be done also on the word space. The following Definition guarantees not, that we have a equivalence relation. Nevertheless we want to use the term \emph{equivalent}.

\begin{defi}
	\label{defi:sim}
	The words $v,w\in\mathcal W$ are said to be equivalent, noted by $v\sim w$, if and only if $|v|=|w|$, $S_v(K)\cap S_w(K)\neq\emptyset$ and $v^-\neq w^-$. Additionally we say, that $v$ is equivalent to itself, such that $v\sim v$ holds.\\
	For $v,w\in\mathcal W^\star$ with $v=v_1v_2\dots$ and $w=w_1w_2\dots$ we extend this relation, such that $v\sim w$, if and only if there exists a $n_0\in\N$ such that $x|_n\sim w|_n$ holds for all $n\geq n_0$.\\
	Further we want to define the number of equivalent words by $R(w)\defgl\#\{v\in \mathcal W: v\sim w\}$.
\end{defi}
For a better understanding are in Figure \ref{fig:CellsSG} some equivalent words highlighted by the same color. 

\begin{kor}
	\label{cor:R(v)}
	If the fractal $K$ fulfills the OSC, then $R(v)<\infty$ for all $v\in\mathcal W$.
\end{kor}

\begin{proof}
	This follows directly by \cite[Prop. 11]{BK1991}.
\end{proof}

\begin{rem}
	\label{rem:gegbsp}
	\begin{figure}[ht]%
		\centering%
		\includegraphics[page=5,width=0.9\textwidth]{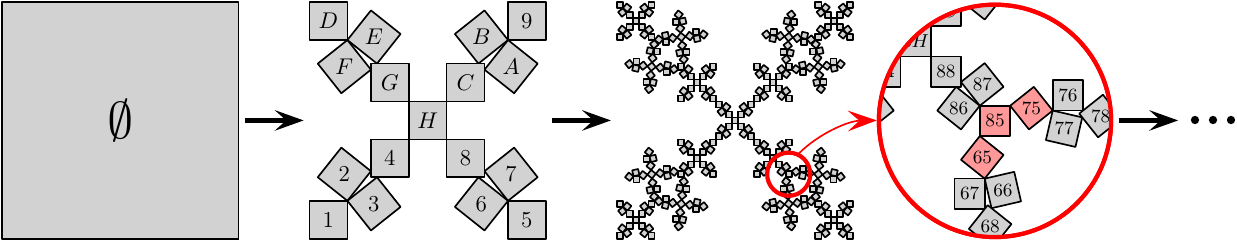}%
		\caption{Construction of the Vicsek-snowflake in the first two steps and a part of the word space of length $2$. In red are three cells highlighted, which disprove, that our relation is always transitive.}%
		\label{fig:gegbsp}%
	\end{figure}%
	The Definition \ref{defi:sim} seems to be quite insufficient, since this does not guarentee, that our relation $\sim$ is transitive. It is a natural question, if this could be induced by some other, more common condition like p.c.f (see \cite{Kigami2001}). Unfortunately, this is not the case.\\
	For this, we construct a counter example as it can be seen in figure \ref{fig:gegbsp}. This IFS contains 17 similitudes and the alphabet consists of $\mathcal A=\{1,\dots,9, A, \dots,H\}$. Each similitude has a contraction ratio of $c\approx 0.1601$%
	, which is the solution of 
	\begin{equation*}-c^2+c(
	 5+\sqrt{2-c^2})-1=0.
	\end{equation*}
	This fractal satisfies the open set condition and is p.c.f.\\
	Now consider the cells $65, 85$ and $75$, which are highlighted in figure \ref{fig:gegbsp}. It holds, that the intersection of the cells $65$ and $85$ is non-empty. Futhermore is the intersection of $85$ and $75$ non-empty and the parents of each words are different. So it follows by Definition \ref{defi:sim} that $65\sim 85$ and $85\sim 75$. On the other hand it holds that $S_{65}(K)\cap S_{75}(K)=\emptyset$ and thus $65\not\sim 75$, which shows that the relation is not transitive and this cannot be induced by a condition like p.c.f..
\end{rem}

We want to take a deeper look at so called nested fractals. As we will see, on those fractals $\sim$ forms a equivalence relation. But first, let us clarify what a nested fractal is.

\begin{defi}[\cite{Hambly2000}]
	\label{defi:nested}
	We want to denote by $F_0'\defgl\{q_i: S_i(q_i)=q_i\}$ the set of all fixed points of the similitudes $S_i$. Further we want do define the set of all essential fixed points $F_0$ by $F_0\defgl\{x\in F_0': \exists i, j\in\mathcal A,  y \in F_0', x\neq y\text{ st. } S_i(x)=S_j(y)\}$.
	A fractal $K$ is then called nested, if it satisfies:
	\begin{enumerate}
		\item Connectivity: For any $1$-cells $C$ and $C'$, there is a sequence $\{C_i: i=0,\dots, n\}$ of $1$-cells such that $C_0=C, C_n = C'$ and $C_{i-1}\cap C_i\neq \emptyset, i=1,\dots, n$.
		\item Symmetry: If $x,y\in F_0$, then reflection in the hyperplane $H_{xy}=\{Z:|z-x|=|z-y|\}$ maps $S^n(F_0)$ to itself.
		\item Nesting: If $v,w\in\mathcal W$ with $v\neq w$, then 
		\begin{equation*}
			S_v(K)\cap S_w(K) = S_v(F_0)\cap S_w(F_0)
		\end{equation*}
		\item Open set condition (OSC): There is a non-empty, bounded, open set $\mathcal O$ such that the $S_i(\mathcal O)$ are disjoint and $\bigcup_{i=1}^NS_i(\mathcal O)\subseteq \mathcal O$.
	\end{enumerate}
\end{defi}

As a first observation we consider some properties of the fixed points $q_i$.

\begin{kor}[{\cite[Corollary in \S 9]{BK1991}}]
	\label{cor:bandt}
	Let $S_1,\dots, S_N$ be a IFS with OSC and attractor $K$. Then belongs $q_i$ exactly to one $S_j(K)$ with $j=i$.
\end{kor}

\begin{kor}
	\label{cor:qiqj}
	Let $S_1,\dots, S_N$ be a IFS with OSC and attractor $K$. Let $q_i$ be the fixed point of $S_i$. If $q_i=q_j$ holds, then $i=j$ follows.
\end{kor}

\begin{proof}
	Consider $q_i=q_j$ with $i\neq j$. By Corollary \ref{cor:bandt} it follows, that $q_i\in S_i(K)$ and on the same time $q_j = q_i\in S_i(K)$ holds. Using again Corollary \ref{cor:bandt}, $i=j$ follows.
\end{proof}

Now we want to analyze the intersection of some cells. The first lemma consider the intersection with a children, the second the intersection of two words with the same length.

\begin{lem}
	\label{lem:231}
	Let $u\in\mathcal W$ and $a\in\mathcal A$. Then it holds:
	The cell $ua$ contains only one element from $S_u(F_0)$, namely $S_u(q_a)$. In other words:
	\begin{equation*}
		S_{ua}(K)\cap S_u(F_0)=S_u(q_a).
	\end{equation*}
\end{lem}

\begin{proof}
	Consider $q_k\in F_0'$. It holds, that $S_k(q_k)=q_k$ is true, thus $S_u(q_k) = S_{uk}(q_k)$ holds and
	\begin{equation}
		S_{ua}(K)\cap S_u(q_k) = S_{ua}(K)\cap S_{uk}(q_k)\label{eq:bew-lemma-231}
	\end{equation}
	follows.\\
	Our goal is to show, that \eqref{eq:bew-lemma-231} is empty for $k\neq a$ and consists of one point, if $k = a$.\\
	So let us take a look at those two cases.\\
	We first consider $k=a$. It follows, that 
	\begin{equation*}
		S_{ua}(K)\cap S_{uk}(q_k)=S_{ua}(K)\cap S_{ua}(q_a)=\{x\}
	\end{equation*}
	holds and thus the intersection consists of one single point.\\
	Let us now consider $k\neq a$. Assume, that \eqref{eq:bew-lemma-231} is non-empty. In particular we assume, that
	\begin{equation*}
		\{y\}= S_{ua}(K)\cap S_{uk}(q_k)= S_{ua}(p)\cap S_{uk}(q_k)
	\end{equation*}
	with $p\in K$. Since $ua\neq uk$ holds by assumption and our fractal is nested, we conclude, that $p\in F_0$ must hold by the nesting property.\\
	Further observe that in general $S_{ua}(F_0)\cap S_{vb}(F_0)\subseteq S_{u}(F_0)\cap S_{v}(F_0)$ holds. Using this it follows, that
	\begin{align*}
		\emptyset &\neq S_{ua}(K)\cap S_{uk}(K) = S_{ua}(F_0) \cap S_{uk}(F_0) \subseteq S_{u}(F_0)\cap S_{u}(F_0) =\\
		&=  S_u(F_0) = \{S_u(q_{b_1}), \dots, S_u(q_{b_n})\} = \{S_{ub_1}(q_{b_1}), \dots, S_{ub_n}(q_{b_n})\}
	\end{align*}
	holds, where $n=|F_0|$ and $b_i\in F_0$ with $b_i\neq b_j$ for $i\neq j$.\\
	Let us denote further
	\begin{equation*}
		\{y_m\}= S_{ua}(q_{c_m})\cap S_{uk}(q_{d_m})\qquad\text{for some } m
	\end{equation*}
	with $q_{c_m}, q_{d_m}\in F_0$.\\
	On the other hand $\{y_m\}=S_{ub_m}(q_{b_m})$ holds. Thus it follows, that $c_m = a$ and $d_m = k$ holds, or more precisely $S_{ua}(q_a) = S_{uk}(q_k)$. We can reformulate this into
	\begin{equation*}
	 S_{u}(q_a) = S_{ua}(q_a) = S_{uk}(q_k) = S_{u}(q_k)
	\end{equation*}
	and see, that $q_a = q_k$ must hold. By Corollary \ref{cor:qiqj} follows $a = k$, which contradicts our assumption.
\end{proof}

\begin{lem}
	\label{lem:232}
	For nested fractals and words $u, v\in\mathcal W$ with $|u| = |v|$, $u\neq v$ and $u^-\neq v^-$ it holds, that $S_u(K)\cap S_v(K)$ consists of at most one single point.
\end{lem}

\begin{proof}
	Let us only consider words $u,v\in\mathcal W$, where $S_u(K)\cap S_v(K)\neq\emptyset$ holds. As a first step we note, that the intersection is a point set, since
	\begin{equation*}
		S_u(K)\cap S_v(K) = S_u(F_0)\cap S_v(F_0) \defgl \{x_1,\dots, x_n\}
	\end{equation*}
	holds by the nesting property.\\
	Now consider $u, v\in\mathcal W$ with $|u|=|v|, u\neq v$ and $u^-\neq v^-$. We want to denote the points $\{x_i\}$ and for this, we can represent them as
	\begin{equation*}
		\{x_i\} = S_{u}(q_{a_i})\cap S_{v}(q_{b_i})
	\end{equation*}
	where $q_{a_i}\neq q_{a_j}$ and $q_{b_i}\neq q_{b_j}$ for $i\neq j$ must hold.\\
	At the same time we can represent this point by
	\begin{equation*}
		\{x_i\} = S_{u^-}(q_{A_i})\cap S_{v^-}(q_{B_i}),
	\end{equation*}
	again with $q_{A_i}\neq q_{A_j}$ and $q_{B_i}\neq q_{B_j}$ for $i\neq j$.\\
	Thus it follows, that $S_u(q_{a_i}) = \{x_i\} = S_{u^-}(q_{A_i})$. We can now apply Lemma \ref{lem:231} and it follows, that $A_i = \tau(u)$ for all $i$ and thus $n \leq 1$ follows.
\end{proof}

Finally we are able to proof the following Proposition:

\begin{prop}
	\label{prop:nestedER}
	It holds, that for nested fractals the relation $\sim$ defined in Definition \ref{defi:sim} is transient and thus defines a equivalence relation.
\end{prop}

\begin{proof}
	It is clear, that $\sim$ is reflexive and symmetric.\\
	The only critical part is the transitivity of $\sim$, thus we want to show, that $u\sim v\sim w$ implies $u\sim w$. For the sake of understanding we want to study $uU\sim vV\sim wW$ with $u,v,w\in\mathcal W$ and $U,V,W\in\mathcal A$.\\
	Without loss of generality we consider $uU\neq vV\neq wW\neq uU$, since it becomes otherwise trivial.\\
	Let us now proof, that $uU\sim wW$ holds. Thus we have to prove, that $|uU|=|wW|$, $S_{uU}(K)\cap S_{wW}(K)\neq \emptyset$ and $u\neq w$.\\
	The first part is very easy, since $|uU| = |vV| = |wW|$ holds.\\
	In our second step we show, that
	\begin{equation*}
		S_{uU}(K)\cap S_{wW}(K) = S_{uU}(F_0)\cap S_{wW}(F_0) \neq\emptyset
	\end{equation*}
	since the fractal is nested. It holds by Lemma \ref{lem:232}, that 
	\begin{equation*}
		S_{uU}(K)\cap S_{vV}(K) = \{x_1\}
	\end{equation*}
	since $uU\sim vV$ and thus $u\neq v$. Further it holds, that
	\begin{equation*}
		\{x_1\}\in S_u(K)\cap S_v(K) = S_u(F_0)\cap S_v(F_0).
	\end{equation*}
	Or in other words: $\{x_1\}\in S_u(F_0)$ and $\{x_1\}\in S_{uU}(K)$. By Lemma \ref{lem:231} it follows, that
	\begin{equation}
		\{x_1\} = S_u(q_U) = S_{uU}(K)\cap S_u(F_0)\label{eq:bew-prop-230-1}
	\end{equation}
	and in the same way
	\begin{equation*}
		\{x_1\} = S_v(q_V) = S_{vV}(K)\cap S_v(F_0).
	\end{equation*}
	On the other hand it follows with the same argumenation for $S_{vV}(K)\cap S_{wW}(K) = \{x_2\}$ that $\{x_2\} = S_v(q_V) = S_{vV}(K)\cap S_v(F_0)$ and 
	\begin{equation}
		\{x_2\} = S_w(q_W) = S_{wW}(K)\cap S_w(F_0)\label{eq:bew-prop-230-2}
	\end{equation}
	holds. In particular we obtain $\{x_2\} = S_v(q_V) = \{x_1\}$ and we conclude, that 
	\begin{equation*}
		\{x_1\} = \{x_2\} \in S_{uU}(F_0)\cap S_{wW}(F_0) = S_{uU}(K)\cap S_{wW}(K) 
	\end{equation*}
	holds. Thus the intersection of $S_{uU}(K)$ and $S_{wW}(K)$ is non-empty.\\
	In our last part we want to prove, that $u\neq w$ holds. For this we want to assume, that $u=w$. It follows, that $U\neq W$ must hold, since otherwise $uU = wW$ would hold. By equation \eqref{eq:bew-prop-230-1} and \eqref{eq:bew-prop-230-2} we know, that 
	\begin{equation*}
		S_u(q_U) = \{x_1\} = \{x_2\} = S_w(q_W) = S_u(q_W)
	\end{equation*}
	holds, where the last equality follows by $u=w$. This implies that $q_U = q_W$ and by Corollary \ref{cor:qiqj} $U = W$ must hold, which contradicts our assumptions. Thus $u\neq w$ follows and we obtain in general, that $\sim$ is transitive.
\end{proof}

\begin{rem}
	\label{rem:SC}
	A general definition of $\sim$ such that $\sim$ is transitive is quite complicated or maybe impossible. It can differ from iterated function scheme to iterated function scheme and in some cases it can be necessary, to adapt the definition to the IFS.\\
	\begin{figure}[ht]%
		\centering%
		\includegraphics[page=2,width=0.775\textwidth]{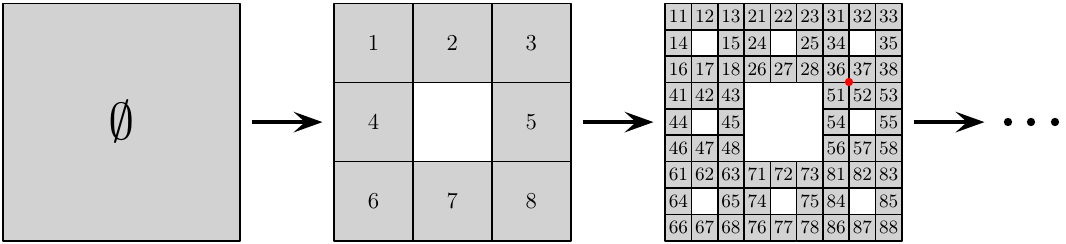}%
		\caption{Construction of the Sierpi\'nski carpet in the first two steps and words up to length $2$. The red point highlights the intersection of the cells $36$ and $52$.}%
		\label{fig:SC}%
	\end{figure}%
	For example, if we take a short look at the Sierpi\'nski carpet as in figure \ref{fig:SC}, we can see that the Definition regarding \ref{defi:sim} won't form a equivalence relation. This is due the fact, that $S_{36}(K)\cap S_{52}(K)$ is non empty (highlighted by the red point). It seems unnatural, that $36\sim 52$ should hold. Instead it should hold, that $36, 51$ and $28$ are equivalent words. The definition of $\sim$ can be futher extended (for example by $25\sim34$) and in this case it holds, that $\sim$ forms a equivalence relation. Further is $R(v)\leq 4$ and we would be able, to describe the Sierpi\'nski carpet with the Markov chain defined below.\\
 	The three words $28, 36$ and $51$ provide a another interesting fact. It holds, that the dimension of those intersections have different Hausdorff dimension: $\dim_H(S_{28}(K)\cap S_{36}(K))=1$, but $\dim_H(S_{28}(K)\cap S_{51}(K))=0$. It is possible, to adjust the transition probability with respect to this fact. By now it is unclear, if this has an effect and what this effect is. We want to study this different question in an other article \cite{FK2019}.
\end{rem}

\begin{assumpt}
	The Remarks \ref{rem:SC} and \ref{rem:gegbsp} as well as Prop. \ref{prop:nestedER} show, that $\sim$ can be a equivalence relation or not, which also depends on the definition of $\sim$. Since this is a essential part of the paper, we make the following assumption for the rest of the paper:
	\begin{center}
	\textbf{(A)} The relation $\sim$ is a equivalence relation
	\end{center}
\end{assumpt}

Additionally we mant to introduce a mass distribution $m$ on the alphabet $\mathcal A$ with $m(a) \in(0,1)$ for all $a\in\mathcal A$ and with $\sum_{a\in\mathcal A} m(a)=1$ and $m(\emptyset)=1$. %
This means, that every similitude gets a probability which we also can understand as every cell becomes a weight. For this reason we can see this as a fractal, where some parts are heavier than others. This should be done iterative, such that we define for a word $w=w_1w_2\dots w_n\in\mathcal W$ with $w_i\in\mathcal A$ the mass $m(w)$ through $m(w)\defgl m(w_1)\cdots m(w_n)$. It holds, that this generates by \cite[Theorem 2.8]{Falconer1997} a self-similar Borel measure $\mu$ such that
\begin{align*}%
 \mu(A) = \sum_{i=1}^Nm(i)\mu(S_i^{-1}(A))
\end{align*}
\begin{figure}[ht]%
	\centering%
	\begin{subfigure}[b]{0.475\textwidth}
		\includegraphics[page=3,width=\textwidth]{Grafiken.pdf}%
		\caption{The weighted Sierpi\'nski gasket with mass of 0.5 in the bottom left triangle. The two other letters have a mass of 0.25.}
	\end{subfigure}
	~
	\begin{subfigure}[b]{0.475\textwidth}
		\includegraphics[width=\textwidth]{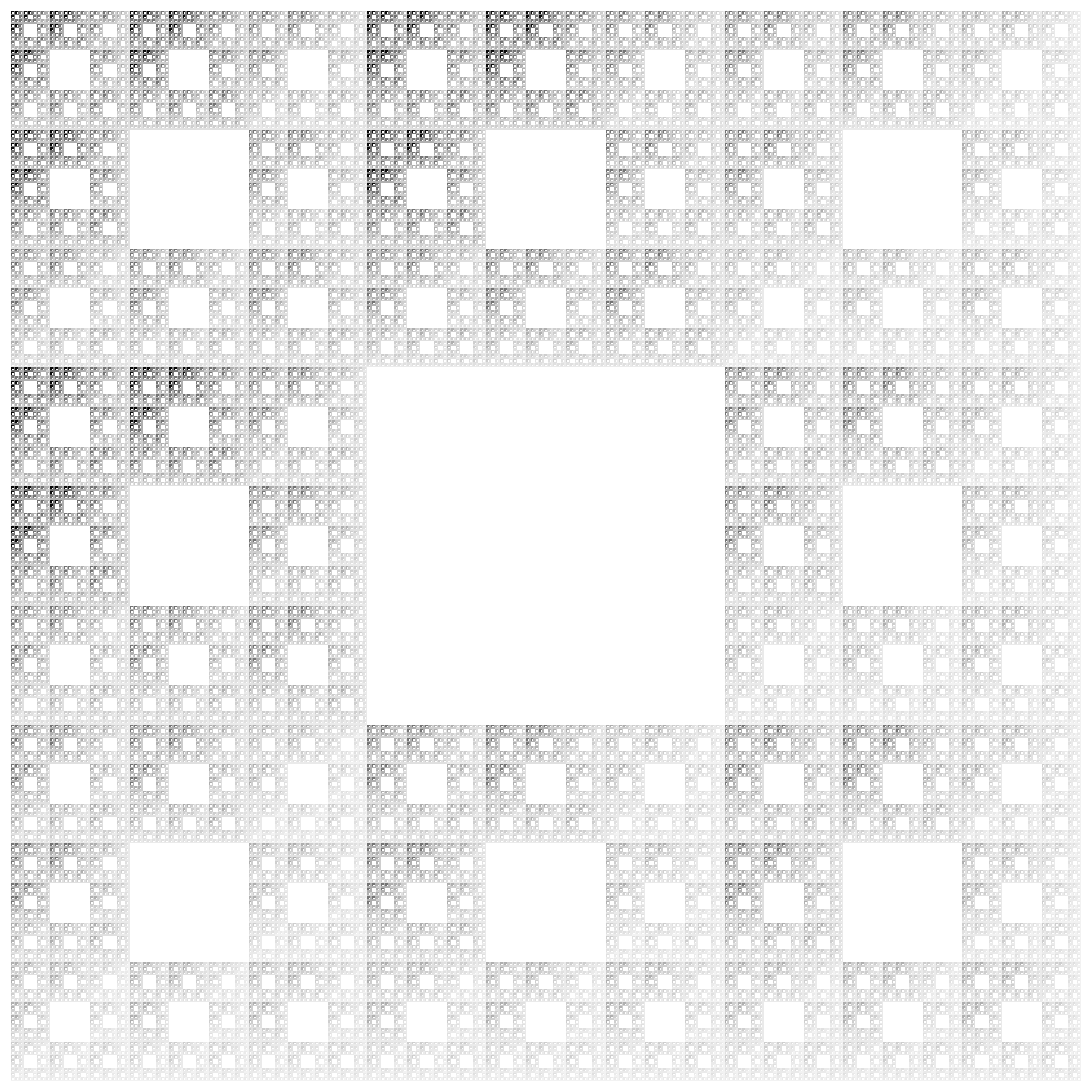}
		\caption{The weighted Sierpi\'nski carpet. The three upper left squares have a mass of 0.2, all other cells have a mass of 0.08.}
	\end{subfigure}
	\caption{Two examples of weighted fractals. A darker color means more weight/mass, a brighter color corresponds to less weight/mass.}%
	\label{fig:gewFraktale}%
\end{figure}%
holds for all Borel sets $A$ with $\operatorname{supp}\mu=K$ and $\mu(K)=1$. In Figure \ref{fig:gewFraktale} are two examples of weighted fractals. Heavier cells are painted in darker color, lighter cells are painted brighter.
\medskip

On the same time we can interpret the weight $m(i)$ as the probability to choose the similitude $S_i$. This leads to an so called probabilistic iterated function scheme. For further informations about (probabilistic) iterated function schemes we refer the reader to \cite{FalconerFG, Falconer1997}.
%
%
%
%
\section{Idea of the transition probability and its consequences}
\label{sec:p}
As a first step we want to define a Markov chain on $\mathcal W$. In order to do this, we have to specify a transition probability $p$ on $\mathcal W\times \mathcal W$. Our purpose is to define the transition probability from one cell to its children in connection with the mass distribution on the alphabet $\mathcal A$.\\
Therefore we consider the idea that the probability of going from $v$ to its child $w$ should be equal to the quotient of the mass in $w$ and the mass we start from, which is the mass of $v$. This means, that we get
\begin{align}%
	p'(v,w)&=\frac{m(w)}{m(v)}\qquad\text{ with $w$ is a child from $v$}\label{eq:HerleitungVonp}
	\intertext{The problem of this definition is indeed, that it would not be a probability measure, since in general $\sum_{w\in\mathcal W}p'(v,w)\neq 1$. Therefore we scale equation \eqref{eq:HerleitungVonp} and get:}
	p(v,w)&=\frac{p'(v,w)}{\sum_{x\in\mathcal{W}}p'(v,x)}\nonumber\\
	\intertext{We now want to clearify, when $w$ is a child from $v$. First, if we have $v\in\mathcal W$, than the word $vi$ with $i\in\mathcal A$ should be a childen of $v$. Second, if we have a conjugated word $\tilde v$ from $v$ we want to identify $\tilde v$ and $v$ as the same. Therefore the children of $\tilde v$, namely by the first thought $\tilde vi$, should also be children of $v$.
	\newline
	In total we get that all children of $v$ are from the shape $w=\tilde vi$ with $\tilde v\sim v$ and $i\in\mathcal A$.
	\newline
	This leads together with equation \eqref{eq:HerleitungVonp} to:}
	&=\frac{\frac{m(w)}{m(v)}}{\sum_{\tilde v\sim v}\sum_{i\in\mathcal A}\frac{m(\tilde vi)}{m(v)}}\nonumber\\
	&=\frac{m(w)}{\sum_{\tilde v\sim v}\sum_{i\in\mathcal A}m(\tilde v)m(i)}\nonumber\\
	&=\frac{m(w)}{\sum_{\tilde v\sim v}m(\tilde v)\sum_{i\in\mathcal A}m(i)}\nonumber\\
	\intertext{since $m$ is a mass distribution over $\mathcal A$, it holds, that $\sum_{i\in\mathcal A}m(i)=1$. So it follows:}
	&=\frac{m(w)}{\sum_{\tilde v\sim v}m(\tilde v)}.\nonumber
\end{align}
In the other case, when $w$ is no child of $v$, we want to set $p(v,w)=0$.\\
\\
This motivates the following Definition.

\begin{defi}%
	\label{def:p}%
	Define the transition probability $p:\mathcal W\times\mathcal W\to[0,1]$ by
	\begin{align*}
		p(v,w)\defgl\begin{cases}
				\frac{m(w)}{\sum_{\hat v\sim v}m(\hat v)},&\text{ if } w=\hat vi\text{ with } \hat v\sim v \text{ and } i \in\mathcal A,\\
				0,&\text{ else.}
			\end{cases}
	\end{align*}
\end{defi}

Using this transition probability we want to denote by $\{X_n\}_{n\geq1}$ the Markov chain on the state space $\mathcal W$. In Figure \ref{fig:pOnSG} this can be seen on the Sierpi\'nski gasket for words up to length 2.
\medskip

\begin{figure}[ht]%
	\centering%
	\includegraphics[page=4,height=15cm]{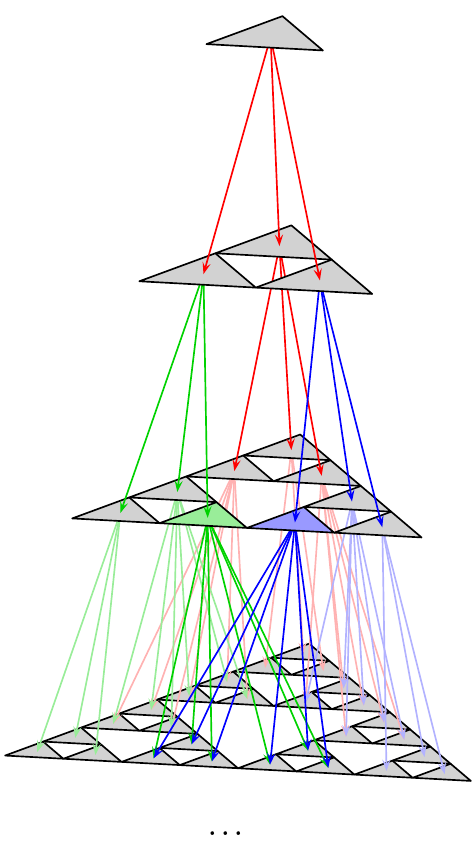}%
	\caption{The transition probability on the Sierpi\'nski gasket.}%
	\label{fig:pOnSG}%
\end{figure}%

Further we want do define the associated Markov operator $P$ by
\begin{align*}%
	(Pf)(v)&\defgl \sum_{w\in\mathcal W}p(v,w)f(w),\qquad v\in\mathcal W
	\intertext{for a nonnegative function $f$ on $\mathcal W$. We call a function $f:\mathcal W\to\mathbb R$ P-harmonic, if}
	(Pf)(v)&=f(v),\qquad v\in\mathcal W
\end{align*}
holds.
\medskip

In order to understand this new definition of the transition probability we first make a observation on a basic property of $p$ in the following Corollary.

\begin{kor}
	\label{kor:pphat}%
	For $v,\tilde v, w\in\mathcal W$ and $\tilde v\sim v$ it holds that
	\begin{align*}
		p(v,w)=p(\tilde v,w).
	\end{align*}
\end{kor}

\begin{proof}
	The Corollary follows by definition and the fact, that $\sum_{u\sim v}m(u) = \sum_{x\sim\tilde v}m(x)$ holds for $\tilde v\sim v$.
\end{proof}

In the next logical step we want to study the transition probability from two arbitrary words $v$ and $w$. The transition probability from $v$ to $w$ is of course only positive, if there is a path between $v$ and $w$ and if $w$ is a successor of $v$.\\
For this, define the $n$-step transition probability from $v\in\mathcal W$ to $w\in\mathcal W$ recursively by
\begin{align*}
	p_n(v,w)\defgl\sum_{u\in\mathcal W}p_{n-1}(v,u)p(u,w),\qquad n\geq 1
\end{align*}
where $p_0(v,w)\defgl\delta_v(w)$ (and where $\delta_v(w)$ is the Kronecker delta function). By obvious reasons it holds, that $p_n(v,w)>0$ only if $d(v,w)=n$ and the following Definition is well defined.

\begin{defi}%
	\label{defi:Green}
	The Green function $g:\mathcal W\times\mathcal W\to \mathbb R$ is defined by 
	\begin{align*}
		g(v,w)\defgl\sum_{n= 0}^\infty p_n(v,w)=p_{|w|-|v|}(v,w),\qquad v,w\in\mathcal W.
	\end{align*}
\end{defi}

Based on the Green function we can observe, if a word $v$ is an ancestor of $w$. Thus we say, that $v$ is an ancestor of $w$, denoted by $v\ll w$, if $g(v,w)>0$ and on the same time we say, that $w$ is a successor of $v$. Further is $v$ a $k$-ancestor of $w$, if $g(v,w)>0$ and $d(v,w)=k$. The set of all $k$-ancestors of $w$ is then defined by \mbox{$\Anc_k(w)\defgl\{v\in\mathcal W:g(v,w)>0\text{ and }d(v,w)=k\}$}. For $w\in\mathcal W^\star$ we define the set of all ancestors through $\Anc(w)\defgl\bigcup_{n=1}^\infty \bigcup_{k=0}^n\Anc_k(w\big|_n)$.
\medskip

These additional notations give us the opportunity to compare our definition of the transition probability with the literature, especially with \cite{LauWang2015}. Since the homogenous case has been already treated, we want to start with a short remark about this case.

\begin{rem}%
	\label{rem:HomogenerFall}
	Definition \ref{def:p} includes the transition probability in the homogenous case, where all weights are equal and thus $m(a)=\frac1N$ for all $a\in\mathcal A$. It follows, that
	\begin{align*}
		p(v,w)=\begin{cases}
				\frac{1}{N\cdot R(v)},&\text{ if } w=\hat vi\text{ with } \hat v\sim v \text{ and } i \in\mathcal A,\\
				0,&\text{ else},
			\end{cases}
	\end{align*}
	since $m(w) = N^{-|w|}$ holds.\\
	The Markov chain with this transition probability is then of DS-type, since all assumptions on $p$ are fulfilled. Transfering the notation of \cite{LauWang2015} to our notation a Markov Chain is of DS-type, if the following five assumptions holding:
	\begin{description}
		\item[(LW1)] $p(v,w)>0$ if $v=w^-$;
		\item[(LW2)] $p(v,w)>0$ implies that either $v=w^-$ or $S_v(K)\cap S_{w^-}(K)\neq\emptyset$;
		\item[(LW3)] $p(v,w)>0$ for any $w$ such that $w^-\sim v$ and $w\sim vk$ for some $k\in\mathcal W$;
		\item[(LW4)] $\inf\{p(v,w):p(v,w)>0, v,w\in\mathcal W\}=: a > 0$;
		\item[(LW5)] there exists a constant $C_0\geq 1$ such that
		\begin{align*}
			\frac{g(\emptyset, w_1)}{g(\emptyset, w_2)}\leq C_0
		\end{align*}
		for any $v\in\mathcal W^\star$ and all $w_1,w_2\in\Anc(v)$ with $|w_1|=|w_2|$.
	\end{description}
	Please note, that Lau and Wang using $v\sim w$ in a slightly other meaning than we do here. They understand by this that two words are neighbors.\\
	This allows us to apply the results of \cite[Theorem 1.2]{LauWang2015} in the homogenous case. It follows, that the (homogenous) Martin boundary is homeomorphic to the self-similar set $K$.
\end{rem}

In the general setting we cannot apply the results of \cite{LauWang2015}, since our Markov chain is not a DS-type Markov chain. The reason for this is, that our transition probability can get arbitrary small and thus does not fulfill assertion (LW4).\\
Therefore we have to consider the Martin boundary theory in total and start with two basic statements.

\begin{lem}[{\cite[Lemma 2.3]{DS2001}}]%
	\label{lemma:DS2001-2.3}
	For any $v,w\in\mathcal W$ and $1\leq k\leq d(v,w)$ we have
	\begin{align*}
		g(v,w)=\sum_{\substack{d(v,u)=k,\\v\ll u\ll w}}g(v,u)g(u,w).
	\end{align*}
\end{lem}

This Lemma was first proven by Denker and Sato and is very useful for us. As a special case we get:

\begin{kor}%
	\label{kor:green}
	For any $v,w\in\mathcal W$ it follows
	\begin{align*}
		g(v,w)=\sum_{u\sim w^-}g(v,u)p(u,w)=p(w^-,w)\sum_{u\sim w^-}g(v,u).
	\end{align*} 
\end{kor}

\begin{proof}
	The first step of the Corollary follows with Lemma \ref{lemma:DS2001-2.3} with $k=d(v,w)-1$ such that the sum is over all $u\sim w^-$. The second step follows by using Corollary \ref{kor:pphat}.
\end{proof}

We want to pick up the idea of the transition probability again but now we take a look at the $n$-step transition probabilities or equivalent the Green function. For this, we consider the mass distribution, which has a multiplicative structure on it and $m(w_1\cdots w_n)=m(w_1)\cdots m(w_n)$ holds. The mass $m(w_i)$ corresponds to the probability of choosing the similitude $S_{w_i}$ and on the same way corresponds $m(w)$ to the probability of choosing the similitude $S_w$. This can be also understood as the probability to pick the cell $w$ starting from $\emptyset$. %
The next theorem proves, that this relation holds.

\begin{theorem}%
	\label{theo:green}
	For all $w\in\mathcal W$
	\begin{align}
		g(\emptyset,w)=m(w).\label{eq:theo:green}
	\end{align}
	holds.
\end{theorem}

\begin{proof}
	We will show the statement by induction over $|w|$ with $w\in\mathcal W$.\\
	Consider the case $|w|=0$. The only word with length $0$ is the empty word $\emptyset$. It follows, that
	\begin{align*}
		g(\emptyset,\emptyset)&=p_0(\emptyset,\emptyset)=1
	\end{align*}
	holds. Further is $m(\emptyset)=1$, so that the statement of the Theorem holds.\\
	Now we take a look on the induction step. The statement \eqref{eq:theo:green} should hold for all words of length $l$. We choose $u\in\mathcal W$ with $|u|=l$ and $i\in\mathcal A$. Now consider $w=ui$ with $|w|=|ui|=l+1$. First we apply 
	Corollary \ref{kor:green} and get:
	\begin{align*}
		g(\emptyset, ui)&=p(u, ui)\sum_{\hat u\sim u}g(\emptyset,\hat u).\\
		\intertext{By induction and definition of $p(u,ui)$ it follows:}
		&=\frac{m(ui)}{\sum_{\hat u\sim u}m(\hat u)}\sum_{\hat u\sim u}m(\hat u)=m(ui).
	\end{align*}
	This proves the statement.
\end{proof}

It seems to be very hard, to calculate $g(v,w)$ for arbitrary $v,w\in\mathcal W$. By calculating some values of $g(v,w)$ one gets the impression, that there is some kind of inner structure which motivates us do define the function $q$.

\begin{defi}%
	Define the function $q:\mathcal W\times\mathcal W\to[0,1]$ by
	\begin{align*}
		q(v,w)=
		\begin{cases}
			\frac{g(v,w)}{m(w)}\sum_{\hat v\sim v}m(\hat v),&\text{ if }v\neq w,\\
			1,&\text{ if }v=w.
		\end{cases}
	\end{align*}
\end{defi}

The function $q$ measures in some sense, how the Green function differs from the quotient of the mass of both points. On the first view this seems to be without benefit. Nevertheless we examine some properties of $q$. As we will see, we are able to calculate the value of $q(v,w)$ independent from $g(v,w)$ by recursion. 

\begin{lem}%
\label{lemma:q-Eigenschaften}
	Let $v,w\in\mathcal W$ and $i,j\in\mathcal A$. The function $q$ fulfills then the following properties:
	\begin{enumerate}[a)]
		\item $g(v,w)=q(v,w)\frac{m(w)}{\sum_{\hat v\sim v}{m(\hat v)}}$ if $v\neq w$,
		\item $q(v,\hat vi)=1$ for $\hat v\sim v$,
		\item $q(v,wi)=\frac{\sum_{\hat w\sim w}q(v,\hat w)m(\hat w)}{\sum_{\hat w\sim w}m(\hat w)}$ if $v\not\sim w$,
		\item $q(v,wj)= q(v,wi)=q(v,\hat wi)=q(v,\hat wj)$ for $\hat w\sim w$.
	\end{enumerate}
\end{lem}

\begin{proof}
	The assertion a) follows by definition, but can be sometimes very useful.\\
	Now we take a look on the other assertions.\\
	We want now to prove $q(v, \hat vi)=1$ and for this, consider $v,\hat v\in\mathcal W$ with $v\sim \hat v$ and $i\in\mathcal A$. It holds, that
	\begin{align*}
		g(v,\hat vi)&=p(v,\hat vi)=\frac{m(\hat vi)}{\sum_{u\sim v}m(u)}.
		\intertext{If we insert this into the definition of $q$, it follows that:}
		q(v,\hat vi)&=\frac{g(v,\hat vi)\sum_{u\sim v}m(u)}{m(\hat vi)}\\
		&=\frac{\frac{m(\hat vi)}{\sum_{u\sim v}m(u)}\sum_{u\sim v}m(u)}{m(\hat vi)}=1
		\intertext{which proves property b).\newline
		Let us now prove assertion c). For this, let $w\in\mathcal W$ and $w\not\sim v$. By definition of $q$ and Corollary \ref{kor:green} it follows:}
		q(v,wi)&=g(v,wi)\frac1{m(wi)}\sum_{\hat v\sim v}m(\hat v)\\
		&=p(w,wi)\sum_{\hat w\sim w}g(v,\hat w)\frac{1}{m(wi)}\sum_{\hat v\sim v}m(\hat v).\\
		\intertext{Inserting the definition of $p(w,wi)$ and using Lemma \ref{lemma:q-Eigenschaften} a) for $g(v,\hat w)$ this leads to:}
		&=\frac{m(wi)}{\sum_{\hat w\sim w}m(\hat w)}\sum_{\hat w\sim w}
		  \frac{q(v,\hat w)m(\hat w)}{\sum_{\hat v\sim v}m(\hat v)}
		  \frac{1}{m(wi)}\sum_{\hat v\sim v}m(\hat v)\\
		&=\frac{\sum_{\hat w\sim w}q(v,\hat w)m(\hat w)}{\sum_{\hat w\sim w}m(\hat w)}.
	\end{align*}
	The property d) follows for $v\not\sim w$ immediately with property c) and for $v\sim w$ with property b).
\end{proof}

In order to prove a strong result on $q$ we first need the following Proposition.

\begin{prop}%
	\label{prop:126}%
	Let $v,w,\hat w\in\mathcal W$ and $\hat w\sim w$. If
	\begin{align}
		w^-\sim (\hat w)^-
		\label{bed:Lemma126}
	\end{align}
	holds, then 
	\begin{align*}
		q(v,w)=q(v,\hat w)
	\end{align*}
	follows.
\end{prop}

\begin{proof}
	For simplicity we write $u\defgl w^-$ so that $w=u\tau(w)$ holds. For $\tilde w\sim w$ we write  $\tilde u=(\tilde w)^-$ so that $\tilde w=\tilde u\tau(\tilde w)$. By assumption it follows, that $\tilde u\sim u$.\\
	By definition it follows that
	\begin{align*}
		q(v,w)&=\frac{g(v,w)}{m(w)\sum_{\hat v\sim v}m(\hat v)}\\
		\intertext{and with Corollary \ref{kor:green} it follows that}
		&=p(u,w)\sum_{\hat u\sim u}g(v,\hat u)\frac1{m(w)\sum_{\hat v\sim v}m(\hat v)}\\
		&=\frac{m(w)}{\sum_{\hat u\sim u}m(\hat u)}\sum_{\hat u\sim u}g(v,\hat u)\frac1{m(w)\sum_{\hat v\sim v}m(\hat v)}\\
		&=\frac1{\sum_{\hat u\sim u}m(\hat u)}\sum_{\hat u\sim u}g(v,\hat u)\frac1{\sum_{\hat v\sim v}m(\hat v)}\\
		\intertext{holds. Since by assumption $\tilde u\sim u$ holds, it follows}
		&=\frac1{\sum_{\hat u\sim \tilde u}m(\hat u)}\sum_{\hat u\sim \tilde u}g(v,\hat u)\frac1{\sum_{\hat v\sim v}m(\hat v)}\\
		\intertext{and by doing the same calulus as before, we get}
		&=\frac{m(\tilde w)}{\sum_{\hat u\sim \tilde u}m(\hat u)}\sum_{\hat u\sim \tilde u}g(v,\hat u)\frac1{m(\tilde w)\sum_{\hat v\sim v}m(\hat v)}\\
		&=p(\tilde u, \tilde w)\sum_{\hat u\sim \tilde u}g(v,\hat u)\frac1{m(\tilde w)\sum_{\hat v\sim v}m(\hat v)}\\
		&=\frac{g(v,\tilde w)}{m(\tilde w)\sum_{\hat v\sim v}m(\hat v)}\\
		&=q(v,\tilde w).
	\end{align*}
\end{proof}

We already made a short precondition in Proposition \ref{prop:126} and we want to introduce a second precondition, such that for all $w\in\mathcal W$ one of both preconditions hold. We want to assume this for the rest of the paper and for a clear structure, we put them in an assumption. To differ those assumptions from \cite{LauWang2015}, we want to note them with a "B" at the beginning.

\begin{assumpt}
	We make the following assumptions for the rest of the paper:
	\begin{description}
		\item[(B1)] The Martin kernel in the homogenous case exists.
		\item[(B2)] For all $w\in\mathcal W$ holds either
			\begin{align*}
				m(w)&=m(\tilde w)~\forall \tilde w\sim w\\
				&\text{or}\\
				w^-&\sim(\tilde w)^-~\forall\tilde w\sim w.
			\end{align*}
	\end{description}
\end{assumpt}

We can easily see, that these assumptions are fulfilled by the Sierpi\'nski gasket and his higher-dimensional analogon. This is very important, otherwise it could be possible, that our assumptions are too restrictive and therefore cannot fulfilled by any fractal.\\
Further is the assumption (B1) in view of Remark \ref{rem:HomogenerFall} unnecessary. Since we didn't calculate the homgenous Martin kernel, we still want to hold on (B1).
\medskip 

Using those assumptions on the structure, we are now able to state and prove the following Theorem.
\begin{theorem}%
	\label{theorem:151}%
	Under assumption (B2) it holds, that 
	\begin{align}
		q(v,wi)&=\frac{1}{R(w)}\sum_{\tilde w\sim w}q(v,\tilde w)\label{eq:theorem:151}.
	\end{align}
	In particular is $q$ independent from $m$.
\end{theorem}

\begin{proof}
	Consider $v,w\in\mathcal W$ with $v\not\sim w$.\\
	If $w\in\mathcal W$ and all $\tilde w\sim w$ fulfill $m(w)=m(\tilde w)$, then it follows with Lemma \ref{lemma:q-Eigenschaften} c):
	\begin{align*}
		q(v,wi)&=\frac{\sum_{\hat w\sim w}q(v,\hat w)m(\hat w)}{\sum_{\hat w\sim w}m(\hat w)}\\
		&=\frac{m(w)}{\sum_{\hat w \sim w}m(w)}\sum_{\hat w\sim w}q(v, \hat w)\\
		&=\frac1{R(w)}\sum_{\hat w\sim w}q(v,\hat w).
		\intertext{If on the other hand holds $w^-\sim(\tilde w)^-$ for all $\tilde w\sim w$, then follows by Proposition \ref{prop:126}, that $q(v,w)=q(v,\tilde w)$ for all $\tilde w\sim w$. Using this, it follows:}
		q(v,wi)&=\frac{\sum_{\hat w\sim w}q(v,\hat w)m(\hat w)}{\sum_{\hat w\sim w}m(\hat w)}\\
		&=\frac{q(v,w)\sum_{\hat w\sim w}m(\hat w)}{\sum_{\hat w\sim w}m(\hat w)}\\
		&=q(v,w)=\frac1{R(w)}\sum_{\hat w\sim w}q(v,\hat w).
	\end{align*}
	The independence of $q$ and $m$ follows immediately through the representation of $q$ by equation \eqref{eq:theorem:151}.
\end{proof}

The fact that $q$ is independent from $m$ is very important and allows us to later, to calculate the Martin kernel in the inhomogenous case.

\begin{bsp}
	As a wide class of examples we want to take a look at the (higher-dimensional) Sierpi\'nski gaskets. In $\mathbb R^2$ this is the \emph{normal} Sierpi\'nski gasket, which we already introduced in figure \ref{fig:CellsSG} and in $\mathbb R^3$ this is the so called Sierpi\'nski tetrahedron. In figure \ref{fig:ST} is the construction of the Sierpi\'nski tetrahedron, where the inner part of each tetrahedron is removed such that it becomes four tetrahedra connected only on the vetrices.\\
	\begin{figure}[ht]%
		\centering%
		\includegraphics[page=1,width=0.9\textwidth]{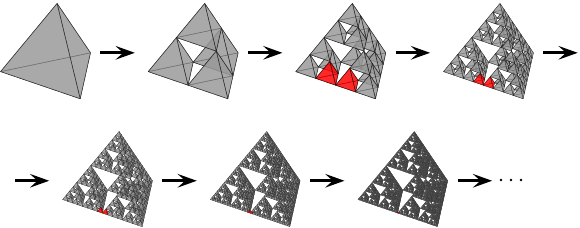}%
		\caption{The construction of the Sierpi\'nski tetrahedron, where two equivalent terrahedra are highlighted.}%
		\label{fig:ST}%
	\end{figure}%
	We want to extend this for every embedding room $\mathbb R^{N-1}$ with dimension $N-1~(N\geq 2)$ and we want to follow mainly the construction in \cite[\S 4]{DS2001}. For this, consider the points $p_1,\dots,p_N\in\mathbb R^{N-1}$. These points should generate a nondegenerate regular simplex $\Delta(p_1,\dots,p_N)\subset\mathbb R^{N-1}$. This means, that the vectors $\overline{p_1p_i}$ (with $i=2,\dots, N$) are linearly independent and the simplex is 
	\begin{equation*}
		\Delta(p_1,\dots,p_N)= \left\{x\in\mathbb R^{N-1}:~x=p_1+\sum_{i=2}^N\lambda_i\overline{p_1p_i},~\lambda_i\geq0,~\sum_{i=2}^N\lambda_i\leq 1\right\}.
	\end{equation*}
	Further we want to define the midpoint of $p_i$ and $p_j$ by $p_{i,j}\defgl \frac{p_i+p_j}2~(=p_{j,i})$. As a next step we want to define the functions of the IFS. For $1\leq k\leq N$ denote by
	\begin{equation*}
		S_{k}:\Delta(p_1,\dots,p_N)\to \Delta(p_{1,k}, \dots, p_{N,k})
	\end{equation*}
	the affine mapping onto the simplex generated by $p_{1,k},\dots,p_{N,k}$ and which satisfies $S_{k}(p_i)=p_{i,k}$. Since $S_k(p_k)=p_{k,k}=p_k$ holds, $p_k$ is a fixed point of $S_k$. For a word $w\in\mathcal W$ we define the iterations of the simplex by
	\begin{equation*}
		\Delta(w)\defgl S_w\left(\Delta(p_1,\dots, p_N)\right).
	\end{equation*}
	The Sierpi\'nski gasket (associated to $p_1,\dots, p_N$) is then defined by
	\begin{equation}
		\mathcal S\defgl \bigcap_{m=1}^\infty\bigcup_{\substack{w\in\mathcal W,\\|w|=m}} \Delta(w).\label{eq:SG-Defi}
	\end{equation}
	We can describe the topology of the Sierpi\'nski gasket by an alphabet $\mathcal A=\{1,\dots,N\}$ with $N$ letters and the corresponding word space. For the equivalence relation we fix $i$ and $j$ and observe, that
	\begin{equation*}
		S_i(p_j)=p_{j,i}=p_{i,j}=S_j(p_i)
	\end{equation*}
	holds. In particular it holds, that $S_j\cap S_i$ is non-empty. As a consequence of this it follows for $u\in\mathcal W$ and $a, b\in\mathcal A$ with $a\neq b$ and $k\geq 1$, that 
	\begin{equation*}
	uab^k\sim uba^k
	\end{equation*}
	holds.\\
	Further we want to assume, that we have a mass distribution $m$ as already introduced in section \ref{sec:preliminaries}.\\
	This allows us to check, if the Sierpi\'nski gasket fulfills assumption (B2). For this, we consider first $w=a^k$ with $a\in\mathcal A$ and $k\geq 1$. In this case assumption (B2) is trivial, since $R(w)=1$.\\
	In all other cases we can describe a word by $w=uab^k$ with $u\in\mathcal W$ and $a,b\in\mathcal A$ with $a\neq b$ and $k\geq1$. The equivalent word $\tilde w\sim w$ can be expressed by $\tilde w=uba^k$. Let us now observe, what happens with different $k$.\\
	In the case of $k=1$ we get, that $m(w)=m(u)m(a)m(b)=m(u)m(b)m(a)=m(\tilde w)$ holds. Thus the first part of (B2) is fulfilled.\\
	For $k\geq 2$ we get, that $w^-=uab^{k-1}\sim uba^{k-1}=(\tilde w)^-$ holds. Thus $w=uab^k$ fulfills the second part of (B2).\\
	In total we get, that (B2) is fulfilled for every $N$. Thus the higher-dimensional Sierpi\'nski gasket is a good example for a fractal, where we can introduce weights and, as we will see later, are able to calculate the Martin kernel and the Martin boundary.
\end{bsp}

%
%
%
%
\section{The Martin kernel}%
\label{sec:k}%
In the next step we want to define the Martin kernel. The Martin kernel is one essential part of the whole Martin boundary theory and is in some sence the regularized Green function.
\begin{defi} 
	The Martin kernel $k:\mathcal W\times\mathcal W\to \mathbb R$ is defined by
	\begin{align*}
		k(v,w)=\frac{g(v,w)}{g(\emptyset, w)},\qquad v,w\in\mathcal W.
	\end{align*}
\end{defi}

One can easily see, that we can also express the Martin kernel by $k(v,w)=\frac{g(v,w)}{m(w)}$ if we apply Theorem \ref{theo:green}. %
Before we continue, we want to examine the Martin kernel and validate some properties of $k$.

\begin{lem}
	\label{lemma:144}
	Let $v,w,\tilde w\in\mathcal W$ with $\tilde w\sim w$ and $i,j\in\mathcal A$. Then it holds:
	\begin{enumerate}[a)]
		\item\label{lemma:144:item-1} $k(v,w)=q(v,w)\frac1{\sum_{\hat v\sim v}m(v)}$ for $v\neq w$,
		\item\label{lemma:144:item-2} $k(w,\tilde wi)=\frac1{\sum_{\hat w\sim w}m(\hat w)}$,
		\item\label{lemma:144:item-3} $k(v,wi)=\frac{\sum_{\hat w\sim w} k(v,\hat w)m(\hat w)}{\sum_{\hat w\sim w}m(\hat w)}$ if $v\not\sim w$,
		\item\label{lemma:144:item-4} $k(v,wi)=k(v,wj)=k(v,\tilde wj)=k(v,\tilde wi)$,
		\item\label{lemma:144:item-5} If $w^-\sim (\tilde w)^-$ for $\tilde w\sim w$ and $w\not\sim v$ holds, then it follows:
		\begin{align*}
			k(v,w)=k(v,\tilde w).
		\end{align*}
	\end{enumerate}
\end{lem}

\begin{proof}
	We consider for all assertions $v,w,\tilde w\in\mathcal W$ with $\tilde w\sim w$ and $i,j\in\mathcal A$.\\
	We first prove assertion \ref{lemma:144:item-1}):
	\begin{align*}
		k(v,w)&=\frac{g(v,w)}{g(\emptyset, w)}
		\intertext{and using Theorem \ref{theo:green} and Lemma \ref{lemma:q-Eigenschaften} a) it follows:}
		&=\frac{q(v,w)\frac{m(w)}{\sum_{\hat v\sim v}m(\hat v)}}{m(w)}\\
		&=q(v,w)\frac1{\sum_{\hat v\sim v}m(\hat v)}.
	\intertext{For statement \ref{lemma:144:item-2}) we use assertion \ref{lemma:144:item-1}), which we just have proven. We get:}
		k(w,\tilde wi)&=q(w,\tilde wi)\frac1{\sum_{\hat w\sim w}m(\hat w)}
		\intertext{and with Lemma \ref{lemma:q-Eigenschaften} b), which states $q(w,\tilde wi)=1$, it follows:}
		&=\frac1{\sum_{\hat w\sim w}m(\hat w)}.
	\intertext{We now take a look at assertion \ref{lemma:144:item-3}). For this, let $v\not\sim w$. Using again statement \ref{lemma:144:item-1}), it follows:}
		k(v,wi)&=q(v,wi)\frac{1}{\sum_{\hat v\sim v}m(\hat v)}
		\intertext{and using Lemma \ref{lemma:q-Eigenschaften} c) we get:}
		&=\frac{\sum_{\hat w\sim w}q(v,\hat w)m(\hat w)}{\sum_{\hat w\sim w}m(\hat w)}\frac1{\sum_{\hat v\sim v}m(\hat v)}
		\intertext{By definition of $q$ it follows:}
		&=\frac{\sum_{\hat w\sim w}\frac{g(v,\hat w)}{m(\hat w)}\sum_{\hat v\sim v}m(\hat v)m(\hat w)}{\sum_{\hat w\sim w}m(\hat w)}\frac1{\sum_{\hat v\sim v}m(\hat v)}\\
		&=\frac{\sum_{\hat w\sim w}g(v,\hat w)\cdot \sum_{\hat v\sim v}m(\hat v)}{\sum_{\hat w\sim w}m(\hat w)}\frac1{\sum_{\hat v\sim v}m(\hat v)}\\
		&=\frac{\sum_{\hat w\sim w}g(v,\hat w)}{\sum_{\hat w\sim w}m(\hat w)}
		\intertext{and with $g(v,\hat w)=k(v,\hat w)m(\hat w)$ it follows:}
		&=\frac{\sum_{\hat w\sim w}k(v,\hat w)m(\hat w)}{\sum_{\hat w\sim w}m(\hat w)}.
	\intertext{Assertion \ref{lemma:144:item-4}) follows also with statement \ref{lemma:144:item-1}):}
		k(v,wi)&=q(v,wi)\frac1{\sum_{\hat v\sim v} m(\hat v)}\\
		\intertext{We use Lemma \ref{lemma:q-Eigenschaften} d) and get:}
		&=q(v,\tilde wj)\frac1{\sum_{\hat v\sim v} m(\hat v)}\\
		\intertext{Using once more assertion \ref{lemma:144:item-1}), we get:}
		&=k(v,\tilde wj).
	\intertext{The proof of assertion \ref{lemma:144:item-5}) uses again assertion \ref{lemma:144:item-1}):}
		k(v,w)&=q(v,w)\frac1{\sum_{\hat v\sim v} m(\hat v)}\\
		\intertext{Through the preconditions we can apply Proposition \ref{prop:126} and it follows, that:}
		&=q(v,\tilde w)\frac1{\sum_{\hat v\sim v} m(\hat v)}\\\\
		&=k(v,\tilde w)
	\end{align*}
	holds.
\end{proof}

We now want to consider fractals which only fulfill our assumptions (B1) and (B2). This means that the Martin kernel can be computed in the homogenous case (for example through the work of \cite{DS2001} or \cite{LauWang2015}) and as a consequence of Theorem \ref{theorem:151} the function $q$ is independent from $m$. For example, the Sierpi\'nski gasket is such a fractal. In this case we can compute the Martin kernel through the homogenous Martin kernel.

\begin{theorem}
\label{theorem:156}
	Denote by $k_{\hom}$ the homogenous Martin kernel. Under assumption (B1) and (B2) it follows, that for $v,w\in\mathcal W$ it holds
	\begin{align*}
		k(v,w)&=
		\begin{cases}
			k_{\hom} (v,w)\frac{R(v)\cdot N^{-|v|}}{\sum_{\hat v\sim v}m(\hat v)},&\text{ for }v\neq w,\\
			\frac{1}{m(v)},&\text{ for }v=w.
		\end{cases}
	\end{align*}
\end{theorem}

\begin{proof}
	Consider the case with $v,w\in\mathcal W$ and $v\neq w$. By Lemma \ref{lemma:144} \ref{lemma:144:item-1}) we get:
	\begin{align}
		q(v,w)&=k(v,w)\sum_{\hat v\sim v}m(\hat v)\label{eq:proof:theorem:156}
		\intertext{since the function $q$ is independent from the mass distribution $m$, equation \eqref{eq:proof:theorem:156} holds for all mass distributions and the value of $q(v,w)$ won't change, if we change the mass distribution. Especially in the homogenous case with $m(\hat v)=m(v)=N^{-|v|}$ we get:}
		&=k_{\hom}(v,w)\sum_{\hat v\sim v}N^{-|\hat v|}\nonumber\\
		&=k_{\hom}(v,w)R(v)N^{-|v|}\label{eq:proof:theorem:156-2}
		\intertext{Now we can use this identity of $q$ and insert this in the general definition of $k$:}
		k(v,w)&=q(v,w)\frac{1}{\sum_{\hat v\sim v}m(\hat v)}\nonumber\\
		&=k_{\hom}(v,w)R(v)N^{-|v|}\frac{1}{\sum_{\hat v\sim v}m(\hat v)}.\nonumber\\
		\intertext{This proves the assertion in the case $v\neq w$.\newline
		We now take a short look, what happens in the case $w=v$. It holds, that}
		k(v,v)&=\frac{g(v,v)}{m(v)}\nonumber
		\intertext{and by definition of $g$ it holds, thats $g(v,v)=\delta_v(v)=1$. Thus it follows:}
		&=\frac{1}{m(v)}.\nonumber
	\end{align}
	Thus the proof of Theorem \ref{theorem:156} is completed.
\end{proof}

This Theorem is essentially the main part of this article. It allows us later to compare the homogenous case with the inhomogenous case. In order to to this, we first have to define a metric on $\mathcal W$.

\begin{defi}
	The Martin metric $\rho$ on $\mathcal W$ is defined by
	\begin{align*}
		\rho(v,w) \defgl \sum_{u\in\mathcal W}a(u)\left|k(u,v)-k(u,w)\right|\text{\qquad for }v,w\in\mathcal W
	\end{align*}
	with $a(u)>0$ for all $u\in\mathcal W$ such that $\sum_{u\in\mathcal W}\frac{a(u)}{g(\emptyset, u)}<\infty$.
\end{defi}

This is indeed a metric. The metric is non-negative, since
\begin{align*}
	\rho(v,w)&=\sum_{u\in\mathcal W}\underbrace{a(u)}_{>0}\underbrace{\left|k(u,v)-k(u,w)\right|}_{\geq 0}\geq0
\end{align*}
and is zero if and only if $v=w$. For this, consider $v=w$, then $\rho(v,w)=\rho(v,v)=\sum_{u\in\mathcal W} a(u)|k(u,v)-k(u,w)|=\sum_{u\in\mathcal W}a(u)\cdot 0=0$.\\
For the reverse conclusion consider $\rho(v,w)=0$. It follows, that $|k(u,v)-k(u,w)|=0$ for all $u\in\mathcal W$ and thus 
\begin{align}
	k(u,v)=k(u,w)\text{ for all }u\in\mathcal W\label{eq:bew:MartinMetrik}
\end{align}
must hold.\\
We assume now, that $v\neq w$. We can split this up in three cases. First we take a look at $|v|<|w|$. If we choose $u=w$, than it follows, that $k(u,v)=k(w,v)=\frac{g(w,v)}{m(v)}=0$, since $g(w,v)=0$. On the other hand it holds, that $k(u,w)=k(w,w)=\frac{g(w,w)}{m(w)}=\frac1{m(w)}\neq 0$. This contradicts equation \eqref{eq:bew:MartinMetrik}.\\
Consider now the case, where  $|v|=|w|$ holds. We choose again $u=w$ and it follows, that $k(u,v)=k(w,v)=\frac{g(w,v)}{m(v)}=\frac{\delta_w(v)}{m(v)}=0$, since by assumption $v\neq w$ holds. At the same time it holds, that $k(u,w)=k(w,w)=\frac1{m(w)}\neq0$ and we get a contradiction to equation \eqref{eq:bew:MartinMetrik}.\\
The last case is $|v|>|w|$. We can choose $u=v$ and in the same way as in case one and it follows, that $k(v,v)=k(u,v)\neq k(u,w)=k(v,w)$, which again contradicts equation \eqref{eq:bew:MartinMetrik}.\\
Further is the metric symmetric, which can be easily seen if we take a look at the definition of $\rho(v,w)$ such that $\rho(v,w)=\rho(w,v)$ holds.
\medskip

The last one is the triangle inequality. This holds, since
\begin{align*}
	\rho(v,w)&=\sum_{u\in\mathcal W}a(u)\left|k(u,v)-k(u,w)\right|\\
	&= \sum_{u\in\mathcal W}a(u)\left|k(u,v)-k(u,x)+k(u,x)-k(u,w)\right|\\
	&\leq\sum_{u\in\mathcal W}a(u)\left(\left|k(u,v)-k(u,x)\right|+\left|k(u,x)-k(u,w)\right|\right)\\
	&= \sum_{u\in\mathcal W}a(u)\left|k(u,v)-k(u,x)\right|+\sum_{u\in\mathcal W}a(u)\left|k(u,x)-k(u,w)\right|\\
	&=\rho(v,x)+\rho(x,w).
\end{align*}
Thus $\rho$ is a metric on $\mathcal W$.

\begin{rem}
	The values $a(u), u\in\mathcal W$ are not needed, but for example can be choosen to be $a(u)\defgl m(u)^{|u|+1}$, which fulfills both conditions on $a(u)$.
\end{rem}
%
%
%
%
\section{The Martin boundary}%
\label{sec:MB}%
We now want to take a look at the Martin boundary. For this, we need to define the completion of $\mathcal W$. We want to examine this in a very precise way to get a precise result.\\
As a first step we want to devote ourself to Cauchy sequences in $\mathcal W$. For this, a sequence $\{w_n\}\subset \mathcal W$ with $|w_n|\to\infty$ is a $\rho $-Cauchy sequence if and only if 
\begin{align*}
	\lim_{n\to\infty} k(v,w_n)\text{ exists for all }v\in\mathcal W. 
\end{align*}
We want to denote the set of all $\rho$-Cauchy sequences by $\widehat{\mathcal W} \defgl \big\{\{w_n\}\subset \mathcal W: \{w_n\}$ is a $\rho$-Cauchy sequence$\big\}$ and we can define a equivalence relation $\eqsim$ on $\widehat{\mathcal W}$ by
\begin{align*}
	\{v_n\}\eqsim\{w_n\} \quad\text{ if and only if }\quad \lim_{n\to\infty} k(u,v_n)=\lim_{n\to\infty} k(u,w_n)\text{ for all }u\in\mathcal W.
\end{align*}%
The equivalence class of $\{w_n\}\in\widehat{\mathcal W}$ will be denoted by $\eqclass{\{w_n\}}$.
It then holds, that the space $\overline{\mathcal W}=\widehat{\mathcal W}\big\slash_{\!\!\eqsim}$ is the collection of all equivalence classes of Cauchy sequences of $\mathcal W$ and is the $\rho$-completion of $\mathcal W$. This space is called the Martin space and is a compact, metric space. We will denote the metric on $\overline{\mathcal W}$ still by $\rho$. The set
\begin{align*}
	\mathcal M\defgl\overline{\mathcal W}\backslash \mathcal W
\end{align*}
is called Martin boundary, which is also a compact metric space, since $\mathcal W$ is open in $\overline{\mathcal W}$. %
For fixed $v\in\mathcal W$ every function $w\mapsto k(v,w)$ can be extended to a continuous function on $\mathcal M$, which we want to denote by $k(v,\cdot)$. For this, let $\xi\in\eqclass{\{w_n\}}\in \mathcal M$ and define
\begin{align*}
	k(v,\xi)=\lim_{n\to\infty}k(v, w_n)\text{\qquad for }v\in\mathcal W.
\end{align*}

As a last point we want to examine the Martin boundary $\mathcal M$ in the inhomogenous case. We want to compare this with the homogenous case and for this, we need to distinguish between the two cases. Therefore we want to denote by $\mathcal M_{\hom}$ the Martin boundary in the homogenous case and in the same way $k_{\hom}(v,w)$, $\rho_{\text{hom}}(v,w)$, $\widehat{\mathcal W}_{\hom}$, $\eqsim_{\hom}$, $\eqclass{\{w_n\}}_{\hom}$ and $\overline{\mathcal W}_{\hom}$. Of course, all properties of the Martin boundary are still valid in the homogenous case.\\
As a preparation of Theorem \ref{theo:MBgleich} we show some useful statements. The first one is about the word space followed by a statement about $\rho$-Cauchy sequences and the equivalence relation.

\begin{kor}%
	\label{kor:Wordspace}%
	The word space $\mathcal W$ is equal to $\mathcal W_{\hom}$.
\end{kor}

\begin{proof}
	This is in fact very easy to see, since the definition of $\mathcal W$ does not depend on the transition probability $p$ and therefore not on the mass distribution $m$.
\end{proof}

\begin{lem}%
	\label{lem:RaumCFGleich}
	Under assumption (B1) and (B2) every $\rho$-Cauchy sequence $\{w_n\}$ is a $\rho_{\hom}$-Cauchy sequence and vice versa. This implies, that $\widehat{\mathcal W}=\widehat{\mathcal W}_{\hom}$ holds.
\end{lem}

\begin{proof}
	Consider a $\rho$-Cauchy sequence $\{w_n\}\in\widehat{\mathcal W}$. Since $|w_n|\to\infty$ there exists a $n_0\in\N$ such that $|w_n|>|u|$ holds for all $n\geq n_0$. For $k(u,w_n)$ it follows with Theorem \ref{theorem:156}, that
	\begin{align*}
		k_{\hom}(u, w_n)&= \frac{\sum_{\hat u\sim u} m(\hat u)}{R(u)\cdot N^{-|u|}} k(u, w_n)
		\intertext{holds for all $n\geq n_0$. So it follows, that}
		\lim_{n\to\infty} k_{\hom}(u,w_n)&=\lim_{\substack{n\to\infty,\\n\geq n_0}} \frac{R(u)\cdot N^{-|u|}}{\sum_{\hat u\sim u} m(u)} k(u, w_n)\\
		&=\frac{R(u)\cdot N^{-|u|}}{\sum_{\hat u\sim u} m(u)} \lim_{\substack{n\to\infty\\n\geq n_0}}k(u,w_n)\text{\qquad exists for all }u\in\mathcal W,
	\end{align*}
	since $\lim_{n\to\infty} k(u,w_n)$ exists for all $u\in\mathcal W$. For this reason $\{w_n\}\in\widehat{\mathcal W}_{\hom}$.\\
	The other way round uses the same argument and completes the proof.
\end{proof}

\begin{lem}%
	\label{lem:ER-ident}
	Under assumption (B1) and (B2) the equivalence relations $\eqsim$ and $\eqsim_{\hom}$ are identical.
\end{lem}

\begin{proof}
	Let $\{v_n\}\in\eqclass{\{w_n\}}\in\widehat{\mathcal W}$. Since $\{v_n\}\eqsim\{w_n\}$ it follows for all $u\in\mathcal W$, that 
	\begin{align}
		\lim_{n\to\infty}k(u,v_n)&=\lim_{n\to\infty}k(u,w_n)\nonumber
		\intertext{holds. Since $|w_n|\to\infty$ and $|v_n|\to\infty$ there exists a $n_0\in\N$ such that $|w_n|>|u|$ and $|v_n|>|u|$ holds for all $n\geq n_0$. It then follows, that }
		\frac{R(u)\cdot N^{-|u|}}{\sum_{\hat u\sim u} m(u)}\lim_{\substack{n\to\infty,\\n\geq n_0}}k_{\hom}(u,v_n)&=\frac{R(u)\cdot N^{-|u|}}{\sum_{\hat u\sim u} m(u)}\lim_{\substack{n\to\infty,\\n\geq n_0}}k_{\hom}(u,w_n)\nonumber
		\intertext{which we can reduce to}
		\lim_{\substack{n\to\infty,\\n\geq n_0}}k_{\hom}(u,v_n)&=\lim_{\substack{n\to\infty,\\n\geq n_0}}k_{\hom}(u,w_n)\nonumber\label{eq:proof:lemmaRel}
	\end{align}
	and we get, that $\{v_n\}\eqsim_{\hom}\{w_n\}$ respectively $\{v_n\}\in\eqclass{\{w_n\}}_{\hom}$ holds.\\
	With the same argument we can show, that $\{v_n\}\in\eqclass{\{w_n\}}_{\hom}$ implies $\{v_n\}\in\eqclass{\{w_n\}}$.
	Overall it follows, that $\eqclass{\{w_n\}}=\eqclass{\{w_n\}}_{\hom}$ holds for all $\{w_n\}\in\widehat{\mathcal W}~(=\widehat{\mathcal W}_{\hom})$.
\end{proof}

All three statements are needed to show our main result.

\begin{theorem}%
	\label{theo:MBgleich}
	The inhomogenous Martin boundary coincides with the homogenous Martin boundary, i.e. $\mathcal M = \mathcal M_{\hom}$.
\end{theorem}

\begin{proof}
	We take a look at the inhomogenous Martin boundary $\mathcal M$. By definition it follows, that
	\begin{align*}
		\mathcal M&=\overline{\mathcal W}\backslash W = \left(\widehat{\mathcal W}\big\slash_{\!\!\eqsim}\right)\backslash\mathcal W
		\intertext{holds. Using Lemma \ref{lem:RaumCFGleich}, we get:}
		&= \left(\widehat{\mathcal W}_{\hom}\big\slash_{\!\!\eqsim}\right)\backslash\mathcal W
		\intertext{Now we can apply Lemma \ref{lem:ER-ident} and it follows, that}
		&= \left(\widehat{\mathcal W}_{\hom}\big\slash_{\!\!\eqsim_{\hom}}\right)\backslash\mathcal W\\
		&=\overline{\mathcal W}_{\hom}\slash \mathcal W
		\intertext{holds. With Corollary \ref{kor:Wordspace} we get}
		&=\overline{\mathcal W}_{\hom}\slash \mathcal W_{\hom}=\mathcal M_{\hom}
	\end{align*}
	which proves the theorem.
\end{proof}

Finally we can compare the inhomogenous Martin boundary with the attractor $K$ of the IFS.

\begin{kor}
	It holds, that
	\begin{equation}
		K\cong \mathcal W^\star\!\big\slash_{\!\!\sim}\cong \mathcal M_{\hom}= \mathcal M.\nonumber
	\end{equation}
\end{kor}

\begin{proof}
	We observered in Remark \ref{rem:HomogenerFall}, that the Markov chain in the homogenous case is of DS-type and therefore fulfills (LW1) - (LW5) from \cite{LauWang2015}. With \cite[Theorem 1.2]{LauWang2015} it follows, that 
	\begin{equation}
	 K\cong \mathcal W^\star\!\big\slash_{\!\!\sim}\cong \mathcal M_{\hom}\nonumber
	\end{equation}
	holds.\\
	The second part follows with Theorem \ref{theo:MBgleich}.
\end{proof}

%
%
%
%
\section{The minimal Martin boundary}
\label{sec:minMB}
In this last section we want to investigate the minimal Martin boundary, also known as space of exits. In a first step we prove, that the function $v\mapsto k(v,\xi)$ is $P$-harmonic. For this, we prove the following helpful Lemma: 
\begin{lem}%
	\label{lem:k(v,w)=k(vi,w)}%
	For any $v,w\in\mathcal W$ it holds that
	\begin{equation*}
		k(v,w)=\sum_{u\in\mathcal W}p(v,u)k(u,w).
	\end{equation*}
\end{lem}

\begin{proof}
	Let $v,w\in\mathcal W$. By definition of the Martin kernel $k$ we get
	\begin{align*}
		k(v,w)&=\frac{1}{g(\emptyset, w)}g(v,w)
		\intertext{We can apply now Lemma \ref{lemma:DS2001-2.3} with $1\leq l\leq d(v,w)$:}
		&=\frac1{g(\emptyset,w)}\sum_{\substack{d(v,u)=l\\v\ll u\ll w}}g(v,u)g(u,w)
		\intertext{We choose $l=1$ and observe, that in this case $g(v, u)=p(v,u)$ holds. This leads to:}
		&=\frac1{g(\emptyset,w)}\sum_{u\in\mathcal W}p(v,u)g(u,w)\\
		&=\sum_{u\in\mathcal W}p(v,u)k(u,w).
	\end{align*}
\end{proof}

\begin{prop}
	\label{prop:201}
	The function $v\mapsto k_\xi(v) \defgl k(v,\xi)$ is $P$-harmonic for every $\xi\in \mathcal M$.
\end{prop}

\begin{proof}
	Consider $v\in\mathcal W$ and $\xi=\eqclass{(w_n)}\in \mathcal M$. It holds:
	\begin{align*}
		(Pk_\xi)(v)&=\sum_{u\in\mathcal W}p(v,u)k_\xi(u)\\
		&=\sum_{u\in\mathcal W}p(v,u)\lim_{n\to\infty}k(u, w_n)\\
		\intertext{$p(v,u)$ is only positive for $u=\tilde vi$ and $\tilde v\sim v, i\in\mathcal A$. Further are only finite summands positive since by Corollary \ref{cor:R(v)} $R(v)<\infty$ and all limits exists. Thus we can interchange limes and summation:}
		&=\lim_{n\to\infty}\sum_{u\in\mathcal W}p(v,u)k(u, w_n)\\
		\intertext{We can now apply Lemma \ref{lem:k(v,w)=k(vi,w)}:}
		&=\lim_{n\to\infty} k(v, w_n)%
		=k_\xi(v)
	\end{align*}
	Thus, the function $k_\xi(\cdot)$ is $P$-harmonic.
\end{proof}

We now want to take a look at the minimal Martin boundary. For this, we recall the Poisson-Martin integral representation, which is one of the nice properties of the Martin boundary. Any non-negative harmonic function $h$ on $\mathcal W$ can be described by
\begin{equation}
	h(\cdot)=\int_{\mathcal M} k(\cdot, y)\mathrm d\mu_h(y)\label{eq:PM-integral}
\end{equation}
with a measure $\mu_h$ on $\mathcal M$, called spectral measure of $h$, which may not be unique.\\
Further the mapping onto the Martin kernel $v\mapsto k_\xi(v)$ (for a fixed $\xi$) can be expressed by \eqref{eq:PM-integral}, since $k_\xi(\cdot)$ is by Proposition \ref{prop:201} harmonic (and non-negative). For shortness we want to denote the spectal measure of $k_\xi(\cdot)$ by $\mu_\xi$.

\begin{defi}
	The minimal Martin boundary $\mathcal M_{\mathrm{min}}$ is defined to be
	\begin{equation*}
		\mathcal M_{\mathrm{min}}\defgl \{\xi\in \mathcal M: \mu_\xi=\delta_\xi\},
	\end{equation*}
	where $\delta_\xi$ is the point mass measure at $\xi$.
\end{defi}
The main purpose of the minimal Martin boundary is, that the spectal measure in \eqref{eq:PM-integral} can be chosen to be supported in $\mathcal M_{\mathrm{min}}$ and is unique. For further informations see for example \cite{Dynkin1969, Woess2009}.

\begin{theorem}
	\label{theorem:202}
	The minimal Martin boundary $\mathcal M_{\mathrm{min}}$ coinsides with the Martin boundary $\mathcal M$.
\end{theorem}

\begin{proof}
	The proof is similar to the proof in \cite{LauWang2015}. Since we modified it slightly, we want to include it.\\
	Our aim is to prove, that $\mu_\xi$ has point mass in $\xi$ for every $\xi\in \mathcal M$.\\
	For this, let $\xi\in \mathcal M$. In a first step we prove
	\begin{equation}
		\mathcal M\backslash\{\xi\} = \bigcup_{u\in\mathcal W : k(u,\xi)=0} \{\zeta\in \mathcal M : k(u,\zeta)>0\}.\label{eq:proof:202}
	\end{equation}
	The inclusion $\supseteq$ is quite simple. For $\zeta\in\bigcup_{u\in\mathcal W : k(u,\xi)=0} \{\zeta\in \mathcal M : k(u,\zeta)>0\}$ it holds, that $\zeta\in \mathcal M$. Suppose, that $\zeta=\xi$. It follows that a $u\in\mathcal W$ exists with $0<k(u,\zeta)=k(u,\xi)=0$. A contradiction and we conclude, that $\zeta\in \mathcal M\backslash\{\xi\}$ holds.\\
	For the other way round consider $\zeta\in \mathcal M\backslash\{\xi\}$. Since $\mathcal M$ is homeomorphic to $\mathcal W^\star\big\slash_\sim$, we can choose $v,w\in\mathcal W^\star$ such that $\{v|_n\}\to\xi$ and $\{w|_n\}\to\zeta$. By assumption it holds, that $\zeta\neq\xi$ and because of this, there exists a $k\in\mathbb N$ such that $v|_k\not\sim w|_k$ holds. The index $k$ marks in this case, where the two words begin to differ from each other.\\
	We can choose $u$ to be $u:=w|_k$. It then holds, that $u\in\mathcal A(w)$, but $u\notin\mathcal A(v)$. It follows, that $k(u,\zeta)>0$ and $k(u,\xi)=0$ and $\zeta$ is part of the right hand side of \eqref{eq:proof:202}.\\
	As a second observation we note, that $\mu_\xi(\{\zeta\in \mathcal M: k(u, \zeta)>0\})=0$ holds for all $u\in\mathcal W$ with $k(u,\xi)=0$. This follows from \eqref{eq:PM-integral}, where
	\begin{equation*}
		0=k_\xi(u)=\int_{\mathcal M} k(u,y)\mathrm d\mu_\xi(y)
	\end{equation*}
	holds and $k(u,y)$ is non-negative for all $u$ and $y$.\\
	In total we get:
	\begin{equation*}
		\mu_\xi(\mathcal M\backslash\{\xi\})=\sum_{u\in\mathcal W:k(u,\xi)=0}\mu_\xi(\{\zeta:k(u,\zeta)>0\})=0
	\end{equation*}
	and thus $\mu_\xi$ is a point mass at $\xi$.
\end{proof}
%
%
%
%
This result is somehow surprising, since one could expect, that the mass distribution changes the Martin boundary. On the other hand describes the Martin boundary mainly the topology of the fractal. The mass distribution has no influence on the topology except the degenerated case with $m(a)=0$ for $a\in\mathcal A$. We excluded this case from the beginning, since we could describe such a fractal using an alphabet with one letter less.
%
%
%
%
\bibliography{Artikel}
\bibliographystyle{alpha}

\end{document}